\documentclass[11pt, oneside]{article}   	
\usepackage{geometry}                		
\usepackage{graphicx}				
\usepackage{color}								
\usepackage{amssymb,amsmath,amsthm}
\newtheorem{thm}{Theorem}
\newtheorem{lem}[thm]{Lemma}

\newtheorem{cor}[thm]{Corollary}

\newtheorem{conj}{Conjecture}
\newtheorem{ex}{Example}

\newtheorem{question}{Question}
\numberwithin{question}{section}

\date{}

\usepackage{graphicx}				
								
\usepackage{amssymb}

\title{Hilbert coefficients of quadratic algebras}
\author{Ralf Fr\"oberg}

\begin{document}
\maketitle

\begin{abstract}
If $R=k[x_1,\ldots,x_n]/I$ is a graded artinian algebra, then the length of $k[x_1,\ldots,x_n]/I^s$ becomes a polynomial in $s$ of degree $n$
for large $s$. If we write this polynomial as $\sum_{i=0}^n(-1)^ie_i{s+n-i-1\choose n-i}$, then the $e_i$'s are called Hilbert coefficients of $I$.
We will study length and Hilbert coefficients of some classes of quadratic algebras.
\end{abstract}

\section{Introduction}
Let $S_n=k[x_1,\ldots,x_n]$, $k$ a field, and let $I$ be a graded ideal in $S_n$, such that $S_n/I$ is Artinian. 
Let $S_n/I^s(t)=\sum_{i\ge0}h_it^i$, where $h_i=\dim_k(S_n/I^s)_i$, be the Hilbert series of $S_n/I^s$ (which is
a polynomial since $S_n/I^s$ is Artinian) and $l(S_n/I^s)=\sum_{i\ge0}h_i$ the length of $S_n/I^s$. Samuel has shown \cite{sa}
that the length of $S_n/I^s$ is a polynomial in $s$ of degree $n$ for all large values of $s$. If we write this polynomial in the form
$$e_0(I){s+n-1\choose n}+\cdots+(-1)^ie_i(I){s+n-i-1\choose n-i}+\cdots+(-1)^ne_n(I),$$
then the $e_i$'s are called the Hilbert coefficients of $I$, and in particular $e_0(I)$ is called the multiplicity of $I$. In fact,
Samuel and others consider the more general case when $I$ is primary for the maximal ideal in a Noetherian local ring, but we stick to the
graded case. The first paper on Hilbert coefficients that we know of is by Northcott \cite{no}. There it is shown that if $I$ 
is an $m$-primary ideal in a CM local ring $(Q,m)$, then
$e_1(I)\ge0$ with equality if and only if $I$ is generated by a system of parameters. Later Narita \cite{nar} showed that with the same conditions $e_2(I)\ge0$
and gave conditions for equality. He also showed that $e_3(I)$ could be negative. Later most articles on Hilbert coefficients has dealt with
connection to depth of the associated graded ring, and with different generalizations of the concept of Hilbert coefficients.

We will study the Hilbert series $R(t)=\sum_{i\ge0}\dim_kR_it^i$ and the Hilbert coefficients for artinian quadratic algebras,
i.e. for rings 
$R=S_n/I^s$, $I$ generated in degree 2. The Hilbert series gives the length as $l(R)=R(1)$. 

\begin{conj}
If $S_n/I$ is a quadratic algebra, then $e_i(I)\ge0$ for all $i$.
\end{conj}

We will start, in Section 2, with  ideals generated by general elements. If $I$ is generated by $n$ general elements, the smallest possible number, then $S/I$ is
a complete intersection, i.e. $I$ is generated by a regular sequence. In that case the Hilbert series of $S_n/I^s$ is more or less known. 
We try to give a bit more explicit expressions for the Hilbert series for quadratic algebras.
If $I$ is generated by more than $n$ general elements, very little is known, and we
give results and conjectures for the Hilbert series of $S_n/I^s$, and thus for the Hilbert coefficients, in some small cases. Then, in Section 3, we
turn to the opposite, namely we consider (quadratic) monomial ideals. We start with the extreme case $I=(x_1,\ldots,x_n)^2$ in $S_n$. 
Here the determination of the Hilbert series 
is trivial, but we have a nice conjecture for the Hilbert coefficients of $I^s$, proved for $n\le11$. Finally we consider all artinian rings with quadratic 
monomial relations in three and four variables.

\section{Ideals generated by general (quadratic) elements}
The following is well known \cite[Theorem1.2]{fi}, \cite[Prop.~6, p.~208]{nag}.

\begin{thm}
Let $f_1,\ldots,f_n$ be a graded regular sequence in $S_n=k[x_1,\ldots,x_n]$, $k$ a field, $\deg(f_i)=d_i$. If $Q=(f_1,\ldots,f_n)$, then
$$l(S_n/Q^s)=l(S_n/Q){s+n-1\choose n}.$$
\end{thm}

\begin{cor}
We have $e_0(S_n/Q^s)=\prod_{i=1}^nd_i$ and $e_i(S_n/Q)=0$ if $i>0$.
\end{cor}

\begin{proof}
This follows from the definition of the $e_i$'s and from $l(S_n/Q)=\prod_{i=1}^nd_i$.
\end{proof}

The graded Betti numbers for $S_n/Q^s$ are described in \cite{gu-vt}.
As a corollary they got an expression for the Hilbert series. (A similar theorem for the Hilbert series exists in \cite[Theorem 4.1]{fi}.) 
We will give a formula for the Hilbert series of $S_n/Q^s$ which is a bit more explicit in the case when $Q$ is generated in degree 2.
We now describe the result on the Hilbert series in \cite{gu-vt}. 

\medskip
Let $ L_{n,s}=\{(a_1,\ldots,a_n)\in\mathbb N^n| \sum_{i=1}^na_i\le s-1\}$. 

\begin{thm}\cite[Corollary 2.3]{gu-vt}
$$H_{S_n/Q^s}(i)=\sum_{(a_1,\ldots,a_n)\in L_{n,s}}H_{S_n/Q}(i-\sum_{i=1}^na_id_i)$$
where $H_R$ is the Hilbert function of $R$.
\end{thm}

We now specialize to $d_i=2$ for all $i$. Then the theorem gives that $H_{S_n/Q^s}(i)=\sum_{(a_1,\ldots,a_n)\in L_{n,s}}H_{S_n/Q}(i-2\sum_{i=1}^na_i)$.
We now suppose that $Q$ is generated by a regular sequence of degree 2 and length $n$ in $S_n=k[x_1,\ldots,x_n]$.

\begin{cor}
The highest degree where $H_{S_n/Q^s}(i)$ is non-zero is $i=2s+n-2$.
\end{cor}

\begin{proof}
$S_n/Q(t)=(1+t)^n$, so it exists in degrees $\le n$. Then the top degree of $S_n/Q^s(t)$ is achieved  at $\max\{ i;i-2\sum_{i=1}^na_i=n\}$, i.e. when
$i=n+2(s-1)$.
\end{proof}

\begin{lem} 
There are ${n+j-1\choose n-1}$ instances for $\sum a_i=j$.
\end{lem}

\begin{cor}
The Hilbert series of $S_n/Q^s$, $Q$ generated by a regular sequence in degree 2 of length $n$ is
$$\sum_{i=0}^{2s-1}{i+n-1\choose n-1}t^i+\sum_{i=2s}^{2s+n-2}\sum_{j\le s-1}{n+j-1\choose n-1}{n\choose i-2j}t^i.$$
\end{cor}

\begin{cor}
${n+2s\choose n}+\sum_{i=2s}^{2s+n-2}\sum_{j\le s-1}{n+j-1\choose n-1}{n\choose i-2j}=2^n{s+n-1\choose n}.$
\end{cor}

To get more explicit expressions for the Hilbert series, we further restrict the number of variables.

\begin{ex} Let $n=2$. Then the Hilbert series of $S_2/Q^s$ is
{\small $$\sum_{i=0}^{2s-1}(s+1)t^i+st^{2s}.$$}
Let $n=3$. Then the Hilbert series of $S_3/Q^s$ is 
{\small $$\sum_{i=0}^{2s-1}{i+2\choose2}t^i+3{s+1\choose2}t^{2s}+{s+1\choose2}t^{2s+1}.$$}
Let $n=4$. Then the Hilbert series of $S_4/Q^s$ is 
{\small $$\sum_{i=0}^{2s-1}{i+3\choose3}t^i+({s+1\choose3}+6{s+2\choose3})t^{2s}+4{s+2\choose3}t^{2s+1}+{s+2\choose3}t^{2s+2}$$}
Let $n=5$. Then the Hilbert series of $S_5/Q^s$ is 
{\small$$\sum_{i=0}^{2s-1}{i+4\choose4}t^i+(5{s+2\choose4}+10{s+3\choose4})t^{2s}+({s+2\choose4}+10{s+3\choose4})t^{2s+1}+5{s+3\choose4}t^{2s+2}+{s+3\choose4}t^{2s+3}.$$}
\end{ex}

Calculating the length we get.

\begin{cor}
{\tiny $${2s+1\choose2}+s=4{s+1\choose2}.$$

$${2s+2\choose3}+4{s+1\choose2}=8{s+2\choose3},$$

$${2s+3\choose4}+{s+1\choose3}+11{s+2\choose3}=16{s+3\choose4},$$

$${2s+4\choose5}+6{s+2\choose4}+26{s+3\choose4}=32{s+4\choose5}.$$}
\end{cor}

We continue to consider ideals generated in degree 2.
Complete intersections are the most general ideals with $n$ generators in $S_n$. We will now look at general ideals $I$
with more than $n$ generators and study $S_n/I^s$. This is much harder, since not even the Hilbert series for $s=1$ is known,
unless $n\le3$ \cite{an}, or we have $n+1$
generators of the ideal \cite{st}. In both these cases the Hilbert series of $S_n/I$, $I$ generated by $r$ generic forms of degree 2,
is $(1+t)^n(1-t^2)^{r-n}$ truncated just before the first nonpositive coefficient. Let $R_{n,r,s}$ be $S_n/I^s$, $I$ generated
by $r$ general elements of degree two. The Hilbert series of $R_{n,r,s}$
is known only for $n=2$ and for $(r,n)=(4,3)$ as far as we know. We have

\begin{thm}\cite[Theorem 8]{bo-fr-lu}
The Hilbert series of $R_{3,4,s}=k[x,y,z]/(f_1,f_2,f_3,f_4)^s$, $f_i$ general forms of degree 2, is
$$\sum_{i=0}^{2s-1}{i+2\choose2}t^i+(3s-1)t^{2s}.$$
\end{thm}

\begin{cor} 
The Hilbert coeficients of $R_{3,4,s}=k[x,y,z]/(f_1,f_2,f_3,f_4)^s$, $f_i$ general forms of degree 2, are
$(e_0,e_1,e_2,e_3)=(8,4,3,4)$.
\end{cor}

It is clear that the Hilbert series of $k[x_1,\ldots,x_n]/I^s$, $I$ generated by $r$ forms of degree 2, is smallest
when the generators are generic. It is easy to find the smallest Hilbert series for $R_{3,5,s}$.

\begin{thm}
The Hilbert series of $R_{3,5,s}=k[x,y,z]/(f_1,f_2,f_3,f_4,f_5)^s$, $f_i$ generic of degree 2, is $\sum_{i=0}^{2s-1}{2i+2\choose2}t^i$ if $s\ge2$.
Thus $(f_1,f_2,f_3,f_4,f_5)^s=(x,y,z)^{2s}$ if $s\ge2$.
\end{thm}

\begin{proof}
It suffices to find an ideal generated by five forms of degree 2 with the claimed Hilbert series. We choose $I=(x^2,y^2,z^2,xy+xz+yz,xz+2yz)$.
A calculation shows that $I^s=(x,y,z)^{2s}$ for $s=3,4,5$. Any $n\ge3$ can be written as $n=3a+4b+5c$ so $I^n=(I^3)^a(I^4)^b(I^5)^c=
m^{3a}m^{4b}m^{5c}=m^n$, where $m=(x,y,z)$.
\end{proof}

\begin{question}
Let $\phi(n)$ be the smallest $s$ such that $I_{r,n}^s=(x_1,\ldots,x_n)^{2s}$. Can one determine $\phi(n)$? Is $\phi(n)=2n-1$?
\end{question}

We have seen the $\phi(3)=5$, and we conjecture that $\phi(4)=7$, $\phi(5)=9$.

\begin{conj}
\begin{itemize}
\item The Hilbert series of $R_{4,5,s}$ is $\sum_{i=0}^{2s-1}{i+3\choose3}t^i+10s^2-25s+35$ if $s\ge4$.

The Hilbert coefficients are $(e_0,e_1,e_2,e_3,e_4)=(16,12,21,35,35).$

\item The Hilbert series of $R_{4,6,s}$ is $\sum_{i=0}^{2s-1}{i+3\choose3}t^i+10s$ if $s\ge4$.

The Hilbert coefficients are $(e_0,e_1,e_2,e_3)=(16,12,1,10).$

\item
The Hilbert series of $R_{4,7,s}$ is $\sum_{i=0}^{2s-1}{i+3\choose3}t^i$ if $s\ge3$.

The Hilbert coefficients are $(e_0,e_1,e_2)=(16,12,1).$

\item
The Hilbert series of $R_{5,9,s}$ is $\sum_{i=0}^{2s-1}{i+4\choose4}t^i$ if $s\ge4$.

The Hilbert coefficients are $(e_0,e_1,e_2)=(32,32,6).$
\end{itemize}
\end{conj}

\section{Monomial ideals}
The smallest ideal generated in degree 2 wich gives an artinian ring is the complete intersection. Now we study the largest ideal
generated in degree 2, namely $(x_1,\ldots,x_n)^2$.
If $I=(x_1,\ldots,x_n)^2$ there is of course no problem to determine $l(S_n/I^s)$, it is ${n+2s-1\choose n}$.
Here the problem is to determine the Hilbert coefficients. We have made extensive calculations, which lead to the following conjectures.

\begin{conj}
If $S_n=k[x_1,\ldots,x_n]$ and $I=(x_1,\ldots,x_n)^2$, then 
\begin{itemize}

\item
$e_i(S_n/I^s)\ne0$ exactly if $i\le\lfloor n/2\rfloor$. 

\item
We have $e_j={n-j\choose j}2^{n-2j}$
if $n\ge 2j$. 

thus
\item
${2s-1+n\choose n}=\sum_{j=0}^{\lfloor n/2\rfloor}(-1)^j2^{n-2j}{n-j\choose j}{s+n-j\choose n-j}.$
\end{itemize}
\end{conj}

\begin{ex}
Let $s=5,n=4$. We have $l(S_4/(x_1,x_2,x_3,x_4)^{2\cdot5})={13\choose4}$, and ${13\choose4}=16{8\choose4}-4\cdot3{7\choose3}+{6\choose2}$.
\end{ex}

We have checked the truth of conjecture for $n\le 11$. It would be nice to have combinatorial proofs.

We continue to consider monomial ideals are generated in degree two. Let $M$ be a set of squarefree
monomials of degree two. We denote by $I_{n,M}$ the ideal
in $S_n=k[x_1,\ldots,x_n]$ which is generated by all squares together with those squarefree which are in $M$. 
If $I$ is a monomial ideal generated in degree two in $S_n$ containing all squares of the variables,
we denote the interesting part of the Hilbert series of $S_n/I^s$, namely $\sum_{i\ge 2s}\dim_kS_n/I^s$, by $H_{s}(S_n/I^s)$. 

\smallskip
We will determine $H_{s}(S_n/I_{n,M}^s)$ and the Hilbert coefficients for all cases in three variables, and conjecture them for all cases in four variables. 
First, if $M=\emptyset$, we get the complete intersection $k[x_1,\ldots,x_n](x_1^2,\ldots,x_n^2)$, and for $M=\{ x_ix_j;\ 1\le i<j\le n\}$ 
we get $k[x_1,\ldots,x_n]/(x_1,\ldots,x_n)^2$.
These rings are already treated.  

Now suppose we have three variables.

\begin{thm}
If $M=\{ x_1x_2\}$, then $H_s(S_n/I_{n,M}^s)=2{s+1\choose2}t^{2s}$ if $n\ge2$. The Hilbert coefficients are $(e_0,e_1)=(8,6)$.

If $M=\{ x_1x_2,x_1x_3\}$, then $H_s(S_n/I_{n,M}^s)=st^{2s}$ if $n\ge2$. The Hilbert coefficients are $(e_0,e_1,e_2)=(8,4,1)$.
\end{thm}

\begin{proof}
If $M=\{ x_1x_2\}$, then the nonzero monomials in $S_3/(x_1^2,x_2^2,x_3^2,x_1x_2)^s$ of degree 2s are the generators of 
$(x_1x_3,x_2x_3)(x_1^2,x_2^2,x_3^2)^{s-1}$,
and $S_3/(x_1^2,x_2^2,x_3^2,x_1x_2)^s$ is 0 in degree $2s+1$. If $M=\{ x_1x_2,x_1x_3\}$, then the nonzero monomials in 
$S_3/(x_1^2,x_2^2,x_3^2,x_1x_2,x_1x_3)^s$ of degree 2s are
the generators of $x_2x_3(x_2^2,x_3^2)^{s-1}$ and $S_3/(x_1^2,x_2^2,x_3^2,x_1x_2,x_1x_3)^s$ is 0 in degree $2s+1$.
\end{proof}

Now we turn to four variables.

\begin{conj}
\begin{itemize}
\ 
\item If $M=\{ x_1x_2\}$, then $H_s(S_n/I_{n,M}^s)=\frac{2s^3+5s^2+3s}{2}t^{2s}+2{s+2\choose3}t^{2s+1}$. The Hilbert coefficients are$(e_0,e_1)=(16,4)$.

 \item If $M=\{ x_1x_2,x_1x_3\}$, then $H_s(S_n/I_{n,M}^s)=4{s+2\choose3}t^{2s}+{s+1\choose2}t^{2s+1}$. The Hilbert coefficients are
 $(e_0,e_1,e_2)=(16,16,2)$.

\item If $M=\{ x_1x_2,x_3x_4\}$, then $H_s(S_n/I_{n,M}^s)=4{s+2\choose3}t^{2s}$. The Hilbert coefficients are
$(16,16,1)$.

\item If $M=\{ x_1x_2,x_1x_3,x_2x_3\}$, then $H_s(S_n/I_{n,M}^s)=\frac{s(s+1)(4s+5)}{6}t^{2s}$. The Hilbert coefficients are
$(e_0,e_1)=(16,8)$.

\item If $M=\{ x_1x_2,x_1x_3,x_1x_4\}$, then $H_s(S_n/I_{n,M}^s)=s(2s+1)t^{2s}+{s+1\choose2}t^{2s+1}$. The Hilbert coefficients are
$(e_0,e_1,e_2)=(16,12,6)$.

\item If $M=\{ x_1x_2,x_2x_3,x_3x_4\}$, then $H_s(S_n/I_{n,M}^s)=s(2s+1)t^{2s}$. The Hilbert coefficients are
$(16,12,5)$.

\item If $M=\{ x_1x_2,x_2x_3,x_3x_4,x_4x_1\}$, then $H_s(S_n/I_{n,M}^s)=2st^{2s}$. The Hilbert coefficients are
$(16,12,1,2)$.

\item If $M=\{ x_1x_2,x_1x_3,x_1x_4,x_2x_3\}$, then $H_s(S_n/I_{n,M}^s)=2{s+1\choose2}t^{2s}$. The Hilbert coefficients are
$(16,12,3)$.

\item If $M=\{ x_1x_2,x_1x_3,x_1x_4,x_2x_3,x_2x_4\}$, then $H_s(S_n/I_{n,M}^s)=st^{2s}$. The Hilbert coefficients are
$(16,12,1,1)$.
\end{itemize}
\end{conj}

\smallskip
The author reports there are no competing interests to declare.

\bigskip
Ralf Fr\"oberg (frobergralf@gmail.com)

Department of mathematics, S-10691, Stockholm, Sweden
\end{document}